\documentclass[sigconf]{acmart}
\renewcommand\footnotetextcopyrightpermission[1]{} 
\usepackage{booktabs} 

\usepackage{algorithmic}
\usepackage{dsfont}
\usepackage{algorithm}

\usepackage[justification=centering]{caption}
\usepackage{subcaption}

\usepackage{enumitem}
\setlist{noitemsep,topsep=0pt,parsep=0pt,partopsep=0pt,leftmargin=*}

\usepackage{multicol}

\theoremstyle{remark}

\newcommand{\bbR}{\mathbb{R}}

\newcommand{\calS}{\mathcal{S}}

\newcommand{\calR}{\mathcal{R}}

\newcommand{\bei}{\begin{itemize}}
\newcommand{\eei}{\end{itemize}}
\newcommand{\bee}{\begin{enumerate}}
\newcommand{\eee}{\end{enumerate}}

\newcommand{\calI}{\mathcal{I}}
\newcommand{\calK}{\mathcal{K}}
\newcommand{\calE}{\mathcal{E}}
\newcommand{\bx}{\mathbf{x}}
\newcommand{\bb}{\mathbf{b}}

\newcommand{\bud}{m}
\newcommand{\INDSTATE}[1][1]{\STATE\hspace{#1\algorithmicindent}}

\copyrightyear{2017} 
\acmYear{2017} 
\setcopyright{acmcopyright}
\acmConference{ADKDD'17}{August 14, 2017}{Halifax, NS, Canada}\acmPrice{15.00}\acmDOI{10.1145/3124749.3124761}
\acmISBN{978-1-4503-5194-2/17/08}

\begin{document}

\title{Optimal  Bidding, Allocation and Budget Spending for a Demand Side Platform Under Many Auction Types}

\author{Alfonso Lobos, Paul Grigas}
\affiliation{%
  \institution{University of California, Berkeley}
}
\email{alobos,pgrigas@berkeley.edu}

\author{Zheng Wen}
\affiliation{%
  \institution{Adobe Research}
}
\email{zwen@adobe.com}

\author{Kuang-chih Lee}
\affiliation{%
  \institution{Alibaba Inc.}
}  
\email{kuang-chih.lee@alibaba-inc.com}

\begin{abstract}
We develop a novel optimization model to maximize the profit of a Demand-Side Platform (DSP) while ensuring that the budget utilization preferences of the DSP's advertiser clients are adequately met. 
Our model is highly flexible and can be applied in a Real-Time Bidding environment (RTB) with \emph{arbitrary} auction types, e.g., both first and second price auctions.
Our proposed formulation leads to a non-convex optimization problem due to the joint optimization over both impression allocation and bid price decisions. Using Fenchel duality theory, we construct a dual problem that is convex and can be solved efficiently to obtain feasible bidding prices and allocation variables that can be deployed in a RTB setting.
With a few minimal additional assumptions on the properties of the auctions, we demonstrate theoretically that our computationally efficient procedure based on convex optimization principles is guaranteed to deliver a globally optimal solution.
We conduct experiments using data from a real DSP to validate our theoretical findings and to demonstrate that our method successfully trades off between DSP profitability and budget utilization in a simulated online environment.
\end{abstract}

%
%


\maketitle

\section{Introduction}\label{sec:intro} 


In targeted online advertising, the main goal is to figure out the best opportunities by showing an advertisement to an online user, who is most likely to take a desired action, such as ordering a product or signing up for an account.  Advertisers usually use the service of companies called demand-side platforms (DSP) to achieve this goal.

In a DSP, each individual advertiser sets up a list of campaigns that can be thought of as plans for delivering advertisements. For each campaign, the advertiser specifies the characteristics of the audience segments that it would like to target (e.g., males, ages 18-35, who view news articles on espn.com) along with the particular media that it would like to display to the target audience (e.g., a video ad for beer). In this work we will call an impression type to a specific collection of those attributes (\textit{e.g.}, male, California, interested in sports). In addition, the advertiser specifies a budget amount, time schedule, pacing details, and performance goals for each campaign. The performance goals typically can be specified by minimizing cost-per-click (CPC) or cost-per-action (CPA). 

The DSP manages its active campaigns for many different advertisers simultaneously across multiple ad exchanges where ad impressions can be acquired through a real-time bidding (RTB) process. In the RTB process, the DSP interacts with several ad exchanges where bids are placed for potential impressions on behalf of those advertisers. This interaction happens in real time when an ad request is submitted to an ad exchange (which may happen, for example, when a user views a news story on a webpage). In this scenario, the DSP needs to offer a solution to decide, among the list of all campaigns associated with its advertiser clients, which campaign to bid on behalf of and the bid values. 

The advertisers who work with the DSP expect its budget to be spent fully or at least in an adequate amount as their marketing areas count on it. Failure to do so may motivate an advertiser to stop working with the DSP in the future, which is unacceptable for its business. In addition, they would like their budget to be spend smoothly if possible. Then, the DSP faces the problem of maximizing its profit while ensuring an adequate budget spending for its advertisers clients. 

DSPs can charge their clients using several pricing schemes, for example in a CPM format advertisers are charged a fixed amount per thousand of impressions showed to users (which is mostly used for branding of products). If the advertisers are interested in some click or action of interest, they may pay in CPM scheme, but requiring that no more than certain amount per click or action of interest is paid (action of interest could be filling a form, purchasing a product, etc.). In this work, we we will assume that the DSP gets paid only when a click or action of interest occurs, but has to pay to the ad exchanges for each impression it acquires. This is a challenging payment setting as the DSP may have a negative operation if the actions or clicks of interest don't occur at the rates the DSP expects. It is important to mention that DSPs  usually receive millions of ad requests opportunities per minute, and their bidding systems needs to respond to each of this ad request in matter of milliseconds making most companies apply simple heuristics to bid in the RTB systems. To simplify notation we will assume in this work that the advertisers are interested in clicks of interest, while this work apply in general to any action of interest. 

As a final remark ad exchanges may use different auctions types to sell advertisement opportunities. As an example, several ad exchanges such as OpenX, AppNexus have announced that they use first price auctions, \textit{i.e.} the highest bidder pays the ad exchange the amount it offered, while others like Google's AdX have announced that they use second price auctions which is that the highest bidder pays the second highest bid submitted to the ad exchange. This add an extra layer to any general DSP optimization algorithm that may want to bid in different ad exchanges for the same advertisers. 



In this paper, we propose a novel approach to maximize the DSP profit while ensuring an adequate budget spending for its advertisers clients. We take into account that the DSP may bid in different ad exchanges who may use different auction rules. Appropriately modeling the impression arrival, auction, and click/action processes our non-convex model gives as an output bidding and allocation vectors that can be used in real time by a DSP to bid in RTB environments. To solve our model we propose a dual formulation using Fenchel conjugates and derive a two-phase primal-dual procedure to solve our non-convex problem. We show that the solutions given by our solution procedure are optimal for several first and second price auctions, results that up to our knowledge are novel in the literature. Experimental results show how our methodology is able to trade off DSP profitability for better budget spending for first and second price auctions in synthetic data, and data based on a real DSP operation.


Due to space limitations we only review works very close to ours, and of those who we take ideas from. In terms of finding biding and allocation schemes different schemes have been suggested in the literature from the ad exchange point of view  \cite{balseiro2014yield,chen2014dynamic,grigas2017profit}, and from the DSP side \cite{perlich2012bid,chen2011real}. In terms of spending the advertisers budget adequately \cite{lee2013real,xu2015smart} set smart pacing strategies. Strategies for bidding using Lagrangian schemes for DSPs have appeared \cite{zhang2014optimal,ren2018bidding} and who use the Ipinyou dataset to validate their results \cite{zhang2014real} as us. Here we formulate a dual problem using the concept  of Fenchel conjugates \cite[p.~91]{boyd2004convex}, which we solve using standard subgradients methods. Our results are similar in spirit to the  recent work \cite{wang2017vanishing}. The latter studies a non-convex multi-agent optimization problem and also uses Fenchel conjugates to construct a dual problem. Our work differs from the latter as we are able to obtain stronger theoretical results in comparison to \cite{wang2017vanishing} using the structure of the online advertising problem studied here (which makes our proofs unique). 


The rest of the paper is organized as follows. 
In Section \ref{sec:model}, we describe the notation and problem statement and we set up the model. Section \ref{sec:optalg} show our proposed optimization problem. In Section \ref{sec:LAD} we propose a dual for our problem of interest, showing several properties of it and proposing a two-phase primal-dual scheme. In Section \ref{Sec:ZD} we show important optimality results and propose two-phase primal-dual scheme to solve our problem of interest. Experimental results using the Ipinyou data \cite{zhang2014real} are presented in Section \ref{sec:expts}, and we mention some future work directions. Section \ref{Sec:Proofs} is the appendix which has the proofs to all propositions and theorems shown in this work.



\section{Model Foundations}\label{sec:model}



Let us begin by describing the basic structure and flow of events in the model. Let $\calK$ denote the set of all campaigns associated with advertisers managed by the DSP.
The DSP interacts with several ad exchanges, and recall that each auction held by one of these ad exchanges represents an opportunity to show an ad to a particular impression (i.e., a user). Although there may be billions of possible impression opportunities each day, we assume that the DSP uses a procedure for mapping each impression opportunity to an \emph{impression type}. Let $\calI$ denote the set of all such impression types. Whenever an opportunity for an impression of type $i \in \calI$ arrives to one of the ad exchanges, the DSP has to make two real-time strategic decisions related to the corresponding auction:  \emph{(i)} how to select a campaign $k \in \calK$ to bid on behalf of in the auction, and \emph{(ii)} how to set the corresponding bid price $b_{ik}$. If the DSP wins the auction on behalf of campaign $k$, then the DSP pays the corresponding market price (which depends on the auction type) to the ad exchange, and an ad from campaign $k$ is displayed to the user. The advertiser corresponding to campaign $k$ is charged only if the user clicks on the ad.
\vspace{-2mm}
\paragraph{Key Parameters for Impression Types and Campaigns}
Our model presumes that the DSP has knowledge (or estimates) of the following parameters:
\begin{itemize}
\item $s_i$ is the expected number of impressions of type $i$ that will arrive during the planning horizon.
\item $\bud_k$ is the (advertiser selected) budget for campaign $k$ during the planning horizon.
\item $\calI_k$ denotes the set of impression types that campaign $k$ targets. (Note that each advertiser can create multiple campaigns to achieve different targeting goals.)
\item $q_k > 0$ is the CPC (cost per click) price for campaign $k$, i.e., the amount charged to the associated advertiser each time a user clicks on an advertisement from campaign $k$.
\end{itemize}
\vspace{-2mm}
\paragraph{Auction Modeling}
We take a flexible approach to auction modeling. In particular, we simply presume that, for each impression type $i \in \calI$, the DSP has constructed the following two bid landscape \cite{cui2011bid} functions (which include first and second price auctions):
\begin{itemize}
\item $\rho_i(b)$ -- the probability of winning an auction for an impression of type $i \in \calI$ given that the DSP submitted a bid of $b$. 
\item $\beta_i(b)$ -- the expected amount the DSP pays the ad exchange, conditional on the DSP winning the auction with a submitted bid of $b$. 
\end{itemize}
We will assume $\rho_i(b)$ and $\beta_i(b)$ to be non-decreasing functions, and $\beta_i(b) \le b$ for all $b \ge 0$ and $i \in \calI$.
\vspace{-2mm}
\paragraph{Click Events}
Whenever an ad of campaign $k \in \calK$ is shown to an impression of type $i \in \calI$ (after the DSP wins the corresponding auction), we presume that a click event happens with probability $\theta_{ik}$. In other words, $\theta_{ik}$ is the expected \emph{click-through-rate}.
%
%
In addition, given an impression type $i \in \calI$ and a campaign $k \in \calK$, let $r_{ik}$ denote the corresponding \emph{expected revenue} earned by the DSP, which is the same as the \emph{expected cost per impression (eCPI)} to the advertiser. Namely, it holds that $r_{ik} := q_k\theta_{ik}$ where $q_k$ is the CPC price defined earlier. 
\vspace{-2mm}
\paragraph{Decision Variables and Additional Notation}
When the DSP has the opportunity to participate in an auction for an impression of type $i \in \calI$ it needs to decide which campaign $k \in \calK$ to bid on behalf of and the bid value to submit. Let $\calE \subseteq \calI \times \calK$ denote the edges of an undirected bipartite graph between $\calI$ and $\calK$, whereby there is an edge $e = (i, k) \in \calE$ whenever campaign $k$ targets impression type $i$, i.e, $\calE := \{(i, k) : i \in \calI_k\}$. Let $\calK_i := \{k \in \calK : (i, k) \in \calE\}$ denote the set of campaigns that target impression type $i$.  For each edge $(i,k) \in \calE$, we define two decision variables:  {\em (i)} $x_{ik}$ the probability that the DSP selects campaign $k$, and {\em (ii)} $b_{ik}$ the bid value to submit to the auction.  Interpreted differently, $x_{ik}$ represents a proportional allocation, i.e., the fraction of auctions for impression type $i$ that are allocated to campaign $k$ on average. (The fraction of impression type $i$ auctions for which the DSP decides to not bid is $1- \sum_{k \in \calK_i} x_{ik}$.) Note that $b_{ik}$ represents the bid price that the DSP submits to an auction for impression type $i$ \emph{conditional} on the fact that the DSP has selected campaign $k$ for the auction.  Let $\bf{x}, \bf{b} \in \bbR^{|\calE|}$ denote vectors of these quantities, which will represent decision variables in our model.

Let us also define some additional notation used herein. For a given set $S$ and a function $f(\cdot) : S \to \bbR$, let $\arg\max_{x \in S} f(x)$ denote the (possibly empty) set of maximizers of the function $f(\cdot)$ over the set $S$. If $f(\cdot) : \bbR^n \to \bbR$ is a convex function then, for a given $x \in \bbR^n$, $\partial f(x)$ denotes the set of subgradients of $f(\cdot)$ at $x$, i.e., the set of vectors $g$ such that $f(y) \geq f(x) + g^T(y - x)$ for all $y \in \bbR^n$. Finally, let $[\cdot]_+$ be the function that returns the maximum between the input and 0, and $'$ denote a derivative in the right context.

\section{Optimization Formulation}
\label{sec:optalg}
Let us begin by recalling the model developed in \cite{grigas2017profit} (proposed only for second price auctions there), which aims to maximize the profit of the DSP under budget constraints:
\begin{align}\label{old-poi}
\underset{ \textbf{x,b}}{\max} &   \sum_{(i,k) \in \calE} [r_{ik} - \beta_i(b_{ik})]s_ix_{ik}\rho_i(b_{ik}) &  \nonumber \\
\text{ subject to } & \underset{i \in \calI_k}{\sum} r_{ik}s_i x_{ik} \rho_i(b_{ik}) \le m_k \ \  \text{for all } k \in \calK \\
 & 0 \le  b_{ik} \le \bar{b}_{i}  \ \ \text{for all } (i,k) \in \calE \nonumber \\
&  \sum_{k \in \calK_i} x_{ik} \le 1 \ \  \text{for all } i \in \calI \nonumber \\
 & x_{ik} \ge 0 \ \ \text{for all } (i,k) \in \calE \nonumber
\end{align}
The first set of constraints above specify that the expected budget spent by each campaign should be less than the total available budget.  The second set of constraints bounds the range of the bid prices, and the third and fourth set of constraints ensure that $\textbf{x}$ represents a valid probability vector when restricted to each impression type. The objective function is the expected DSP profit, which we aim to maximize. Indeed, note that for each pair $(i,k) \in \calE$ the quantity $r_{ik} - \beta_i(b_{ik})$ is the expected profit earned by the DSP whenever an ad of campaign $k$ is show to an impression of type $i$, and $s_ix_{ik}\rho_i(b_{ik})$ is the expected number of impression of type $i$ that we will acquire on behalf of campaign $k$.
Therefore, $[r_{ik} - \beta_i(b_{ik})]s_ix_{ik}\rho_i(b_{ik})$ is the expected profit due to bidding for impressions of type $i$ on behalf of campaign $k$, and summing these quantities over all pairs $(i,k) \in \calE$ yields the total expected profit for the DSP, which we call $\pi(\bx,\bb)$.

Notice that the previous formulation does \emph{not} ensure or even encourage an adequate budget spending for the campaigns, it only ensures that each campaign does not spend in expectation more than its total budget. 
In reality, advertisers are not satisfied by merely ensuring that their spending on each campaign is below the specified budget level. Rather, most advertisers view the budget value $m_k$ as a ``target'' and may have complex preferences regarding their spending behaviors. For example, an advertiser may be very dissatisfied with underspending behavior and may in fact prefer slightly overspending above the budget value $m_k$ instead of severely underspending well below $m_k$. 

In order to greatly enhance the flexibility of our model as well as its ability to capture complicated budget spending preferences, we replace the budget constraints in \eqref{old-poi} with a more general utility function model as follows. First, note that the expected total spending of campaign $k \in \calK$, as a function of the decision variables, is given by $v_k(\bx, \bb) := \underset{i \in \calI_k}{\sum} r_{ik}s_i x_{ik} \rho_i(b_{ik})$. Now, let $u_k(\cdot) : \bbR_{+} \to \bbR$ be a concave utility function describing the budget spending preferences of campaign $k$, whereby $u_k(v_k)$ is the ``utility" of campaign $k$ when its expected spending level is $v_k$. Furthermore, define the vector of expected spending levels $v(\bx, \bb) \in \mathbb{R}^{|\calK|}$ whose $k^{\text{th}}$ coordinate is $v_k(\bx, \bb)$, and let $u(\cdot) : \mathbb{R}^{|\calK|} \rightarrow \mathbb{R}$ be the overall budget spending utility function whereby $u(v(\bx, \bb)) = \sum_{k \in \calK} u_k(v_k(\bx, \bb))$. Finally, as an extra way to simplify notation let's define the feasible set of allocation and bidding variables:
$$\calS := \left\{(\bx, \bb) : \sum_{k \in \calK_i} x_{ik} ~\leq~ 1 \ \text{for all } i \in \calI, \ 0 \le b_{ik} \le \bar{b}_{i} \ \text{for all } (i,k) \in \calE, \ \bx \geq 0\right\} \  $$

We are now ready to state our proposed optimization model: 
\begin{align}
F^* := \underset{ (\textbf{x,b}) \in S}{\max}    \sum_{(i,k) \in \calE} [r_{ik} - \beta_i(b_{ik})]s_ix_{ik}\rho_i(b_{ik}) + u(v(\bx, \bb)) \label{poi_deterministic}
\end{align}
Note that problem \eqref{poi_deterministic} is non-convex, and in Section \ref{sec:LAD} we propose a computationally efficient procedure based on convex duality. We finish this section by showing three examples of utility functions that illustrate the improved generality and flexibility of model \eqref{poi_deterministic}.
\begin{description}\label{Ex:UF}
\item[Examples of Utility Functions]
\end{description}
\begin{enumerate}
  \item Formulation \eqref{old-poi} may be recovered as a special case of the more general problem \eqref{poi_deterministic} by letting $u_k(\cdot)$ be the (extended real valued) concave function such that $u_k(v_k)$ equals $-\infty$ if $v_k$ is strictly greater than $m_k$, and $0$ otherwise.
  \item If we want to maximize the DSP profit but also try to enforce an appropriate target spending for a campaign $k \in \calK$, we can take $u_k(\cdot)$ to be the concave function such that $u_k(v_k)$ equals $-\infty$ if $v_k$ is strictly greater than $m_k$, and $-\frac{\tau_k}{2} \lVert v_k - m_k \rVert_2^2$ otherwise. Here $\tau_k \ge 0$ is a user defined penalization constant. 
  \item If we want to maximize the DSP profit while requiring both a minimum and maximum expected spending for campaign $k \in \calK$, we can take $u_k(\cdot)$ to be the concave function such that $u_k(v_k)$ equals $-\infty$ if $v_k$ is strictly greater than $m_k$ or strictly less than $\alpha_k m_k$, and $0$ otherwise. Here, the parameter $\alpha_k \in [0,1]$ is user defined and represents the minimum percentage of expected budget spending. 
\end{enumerate}
Note that the model \eqref{poi_deterministic} allows for each campaign to have its own distinct utility function $u_k(\cdot)$, and therefore the three examples above may be combined together across the different campaigns, for example. Finally, note also that the separable structure of the utility function $u(\cdot)$, whereby $u(\cdot) = \sum_{k \in \calK} u_k(\cdot)$, is actually not needed for all of the results that we develop herein. Indeed, the only crucial assumption about $u(\cdot)$ is that $u(\cdot)$ is a concave function. However, the separable structure is quite natural and all of our examples do have this separable structure as well, so for ease of presentation we present the model in this way.

\section{Dual Optimization Problem and Scheme}\label{sec:LAD}
We begin this section with a high-level description of our approach for solving \eqref{poi_deterministic}. Our algorithmic approach is based on a two phase procedure. In the first phase, we construct a suitable dual of \eqref{poi_deterministic}, which turns out to be a convex optimization problem that can be efficiently solved with most subgradient based algorithms. A solution of the dual problem naturally suggests a way to set the bid prices $\mathbf{b}$. In the second phase, we set the bid prices using the previously computed dual solution then we solve a convex optimization problem that results when $\bb$ is fixed in order to recover allocation probabilities $\bx$.  A very mild assumption we make for the rest of the paper is that $F^*>-\infty$, otherwise there would be no optimization problem to optimize. 

As we mentioned before we have assumed that our utility function $u(\cdot)$is concave, therefore $-u(\cdot)$ is a convex function and we can define $p(\cdot): \mathbb{R}^{|\calK|} \rightarrow \mathbb{R}$ its convex conjugate as $p(\lambda) : = \sup_{v \in \mathbb{R}^{|\calK|}} \left\lbrace \lambda^Tv + u(v) \right\rbrace$ which is a convex function. The Fenchel Moreau Theorem \cite[p.~91]{boyd2004convex} ensures that $u(v) = \inf_{z \in \mathbb{R}^\calK} \left\lbrace  p(z) - z^Tv \right\rbrace$. Using the latter we can re-write our primal problem as (here we use $\pi(\bx,\bb):= \sum_{(i,k) \in \calE} [r_{ik} - \beta_i(b_{ik})]s_ix_{ik}\rho_i(b_{ik})$): 
\begin{align}
F^* : = \underset{ (\textbf{x,b}) \in S }{\max}   F(\textbf{x,b}) : = \left\lbrace \pi(\bx,\bb) + \underset{ \lambda \in \mathbb{R}^\calK }{\inf} \left( - \lambda^T  v(\bx, \bb) +p(\lambda) \right) \right\rbrace  \nonumber
\end{align}
For a given $\lambda \in \mathbb{R}^{|\calK|}$ we  will define the dual function as:
\begin{align*}
Q(\lambda) : =  \underset{ (\textbf{x,b}) \in S }{\sup} \left( \pi(\bx,\bb) -   \lambda^T v(\bx, \bb)  \right) +p(\lambda) \nonumber
\end{align*}
And the dual problem as:
\begin{align*}
Q^* := \underset{ \lambda \in \mathbb{R}^{|\calK|} }{ \min}  \ \ Q(\lambda)
\end{align*}
Then, the following inequalities hold for our primal and dual formulations (they follow from the max-min inequality  \cite[p.~238]{boyd2004convex}):
\begin{align*}
Q(\lambda) \ge Q^* \ge F^* \ge F(\textbf{x,b}) \ \ \text{for all } \lambda \in \mathbb{R}^{|\calK|}, (\textbf{x},\textbf{b}) \in S 
\end{align*}
Let's now define $h_i(z,b):= (z-\beta_i(b)\rho_i(b)$ the expected profit the DSP receives from showing an ad of campaign $k$ to an impression of type $i$ which has an expected revenue of \$$z$ when submitting a bid of value \$$b$. Then, given a fixed $\lambda \in \mathbb{R}^{|\calK|}$ the dual function $Q(\lambda)$ can re-defined as:
$$Q(\lambda) : = \underset{(\textbf{x},\textbf{b}) \in S}{\textrm{maximize}} \ \ \sum_{k \in \calK} \sum_{i \in \calI} h_i(r_{ik}(1-\lambda_k),b)s_ix_{ik} +p(\lambda)$$
\begin{proposition}[Efficient computation of $Q(\lambda)$]\label{Prop:MaxQ}
Given a dual variable $\lambda$, an optimal solution $(\textbf{x}(\lambda),\textbf{b}(\lambda))$ for the dual function $Q(\lambda)$ can be found using Algorithm \ref{Alg:MaxQ}.
\end{proposition}
\floatname{algorithm}{Algorithm} 
\begin{algorithm} 
\caption{Computing $(\textbf{x}(\lambda),\textbf{b}(\lambda)) \in \arg Q(\lambda)$ for a fixed $\lambda$} \label{Alg:MaxQ}
\begin{algorithmic}
\STATE {\bf Input:} $\lambda \in \bbR^{|\calK|}$.\\
\STATE 1. For each $(i,k) \in \calE$,  set:
\vspace{-1mm}
\begin{equation*}
\begin{split}
&b(\lambda)_{ik} \gets \arg\max_{b \in [0,\bar{b}_i]} \ \ h_i(r_{ik}(1-\lambda_k),b)  , \text{ and } \\
\vspace{-3mm}
&\pi(\lambda)_{ik} \gets h_i(r_{ik}(1-\lambda_k),b(\lambda)_{ik}) s_i
\end{split}
\end{equation*}
\vspace{-2mm}
\STATE 2. For each $i \in \calI$, pick $k_i^* \in \underset{k \in \calK_i}{\max} \{ \pi(\lambda)_{ik} \}$ arbitrary. Set $x(\lambda)_{ik}=0$ for all $k \in \calK_i \ne k_i^*$, and if $\pi(\lambda)_{ik_i^*}>0$ make $x(\lambda)_{ik_i^*} = 1$ otherwise $x(\lambda)_{ik_i^*} = 0$.    
\STATE {\bf Output:} $(\textbf{x}(\lambda),\textbf{b}(\lambda))$ .
\end{algorithmic}
\end{algorithm}
Theorem \ref{Alg:MaxQ} shows that the dual problem can be solved in a parallel fashion, and furthermore finding $\bb(\lambda)$ can be a simple operation. For example, in  the case of a second price auction it is known that $[r_{ik} (1-\lambda_k)]_+ \in \arg\max_{b \in [0,\bar{b}_i]} \ \ h_i(r_{ik}(1-\lambda_k),b)$, and some examples for first price auctions have nice close forms as shown in the next section. Being able to solve $Q(\lambda)$ efficiently is of great importance as it is a core component to find a subgradient of $Q(\lambda)$ as the following theorem shows:
\begin{proposition}\label{Prop:SubgQ}
Given $\lambda \in \mathbb{R}^{|\calK|}$ the output of Algorithm \ref{Alg:SubgQ} is a vector $g \in \partial Q(\lambda)$.
\end{proposition}
\floatname{algorithm}{Algorithm}
\begin{algorithm}
\caption{Computing a subgradient of $Q (\lambda)$} \label{Alg:SubgQ}
\begin{algorithmic}
\STATE {\bf Input:} $\lambda \in \bbR^{|\calK|}$ \\
\STATE 1. Obtain $(\textbf{x}(\lambda),\textbf{b}(\lambda) ) \in \arg\max Q(\lambda)$ using Algorithm \ref{Alg:MaxQ}.
\STATE 2. Obtain $p' \in \partial p(\lambda)$
\STATE 3 Set:\vspace{-1mm}
$$g(\lambda) \leftarrow p' - v(\textbf{x}(\lambda),\textbf{b}(\lambda))$$
\vspace{-3mm}
\STATE {\bf Output:} $g(\lambda) \in \partial Q(\lambda)$ .
\end{algorithmic}
\end{algorithm}
Notice that with STEP 1. in Algorithm \ref{Alg:MaxQ} we can recover bidding prices $\bb(\lambda)$ from dual variables $\lambda \in \mathbb{R}^{|\calK|}$. The final Theorem of this section shows that for fixed bidding prices $\bb$ is easy to obtain allocation probabilities $\bx$ by solving problem \eqref{poi_deterministic}.
\begin{proposition}\label{Prop:FbConvex}
For fixed bidding prices $\bb$, problem \eqref{poi_deterministic} is a convex problem.
\end{proposition}
Proposition \ref{Prop:FbConvex} tell us that we could use a sub-gradient method to find to find an allocation vector $\bx$ given a fixed $\bb$. Better than the previous, depending on the utility function used \eqref{poi_deterministic} can have a nice structure, for example for the utility function examples shown in the previous section, examples 1. and 3. transform problem \eqref{poi_deterministic} in a linear program and example 2. in a quadratic problem. Problems that could be solved directly using solvers like Gurobi \cite{gurobi}. We finish this section by presenting Algorithm \ref{Alg:TwoPhase} which formalize our approach to solve problem \eqref{poi_deterministic}.\vspace{-1mm}
\begin{algorithm}\caption{Two Phase primal-dual Scheme}\label{Alg:TwoPhase}
\begin{algorithmic}
\STATE {\bf Phase 1: Solve the Dual problem.}
\INDSTATE Solve $\min_{\lambda} Q(\lambda)$ to near global optimality 
\INDSTATE using a subgradient method obtaining  dual variables $\hat\lambda$.
\STATE {\bf Phase 2: Primal Recovery}
\INDSTATE 1. For all $(i,k) \in \calE$ do: 
$$\hat{b}_{ik} \gets \arg\max_{b \in [0,\bar{b}_i]} \ \ h_i(r_{ik}(1-\lambda_k),b)$$ 
\vspace{-2mm}
\INDSTATE 2. Using bid prices $\bb(\lambda)$ solve \eqref{poi_deterministic} obtaining  allocation
\INDSTATE probabilities $\hat{\bx}$.
\STATE {\bf Output:} Feasible primal solution $(\hat{\bx}, \hat{\bb})$.
\end{algorithmic}
\end{algorithm}

\section{Zero Duality Gap Results}\label{Sec:ZD}
Algorithm \ref{Alg:TwoPhase} can always be used as long as the parameters and functions of problem \eqref{poi_deterministic} are well defined. Here we we will go further and show that our dual formulation and dual scheme are the right methods to solve \eqref{poi_deterministic}. In particular, we have strong duality results which to the best of our knowledge are novel and have important applications to first and second price auctions by showing optimal bidding prices to be used in an RTB environment. These will be derived from the following theorem:
\begin{theorem} \label{Thm:Applied}
If for all $i \in \calI$ we have that $\rho_i(\cdot)$ and $\beta_i(\cdot)$ are differentiable and:
\begin{itemize}
\item $\rho_i'(b)>0$ for all $b \in [0,\bar{b}_i]$.
\item $g_i(b) = \frac{(\rho_i(b) \beta_i(b))'}{\rho_i'(b)}$ is strictly increasing for all $b \in [0,\bar{b}_i]$.
\end{itemize}
for all $i \in \calI$, then for any $\lambda^* \in \arg\underset{\lambda \in \mathbb{R}^{|\calK|}}{\max}$ $Q(\lambda)$, we have $F(\bb^*,\bx^*)=Q(\lambda^*)$ for any $(\bx^*,\bb^*) \in \arg\underset{(\bb,\bx) \in S}{\max} Q(\lambda^*)$. 
\end{theorem}
Notice that Theorem \ref{Thm:Applied} ensures that no duality gap exists, but furthermore for an optimal dual variable $\lambda^*$ it gives a form of the variables $(\bx^*,\bb^*)$ such that $F(\bx^*,\bb^*)=Q(\lambda^*)$. Also, notice that the second condition of the theorem is a form of diminishing returns. 
\begin{description}
    \item[Applications of Theorem \ref{Thm:Applied}]
\end{description}
\begin{enumerate}
  \item If for all $i \in \calI$ their auctions are second price and $\rho_i(\cdot)$ is a strictly increasing function in $[0,\bar{b}_i]$, then Theorem \ref{Thm:Applied} holds.  Also, for an optimal dual variable $\lambda^*$ optimal bidding prices are $b(\lambda)_{ik}=[\min\{\bar{b}_i,r_{ik}(1-\lambda_k^*)\}]_+$ for all $(i,k) \in \calE$.
  \item If for all $i \in \calI$ all auctions are first price auctions or more generally scaled first price auctions in which the winning DSP pays an $\alpha\in (0,1]$ percentage of the bid it offered, Theorem \ref{Thm:Applied} holds when $\rho_i(\cdot)$ is a strictly increasing twice differentiable concave function. Example of the later are: 1) $\rho(b)=\tfrac{b}{c+b}$ for some fixed $c>0$, 2) square roots and logarithms, 3) $\rho(\cdot)$ representing the cumulative distribution function of an exponential or logarithm-exponential random variable.
  \item If for all $i \in \calI$ all auctions are first scaled price auctions with $\alpha \in (0,1]$ in which for each impression $i$ there is a fixed number $n_i \ge 1 $ of other DSPs who bid independently as $Uniform(0,\bar{b}_i)$, then Theorem \ref{Thm:Applied} holds. Furthermore, for an optimal dual variable $\lambda^*$ optimal bidding prices are $b(\lambda)_{ik}= \left\lbrack \min \left\lbrace \bar{b}_i,\tfrac{n_i r_{ik}(1-\lambda_k^*)}{\alpha (n_i+1)}  \right\rbrace \right\rbrack_+$ for all $(i,k) \in \calE$.
  \item Any combination of the above, or cases in which each impression type satisfies the conditions of Theorem \ref{Thm:Applied}.
\end{enumerate}

We finish this section by making three important comments. First, to obtain the form of the optimal bidding prices is only needed to solve $\arg \max_{b \in [0,\bar{b}_i]}$ $h_i(r_{ik}(1-\lambda_k^*),b)$. In many cases, like second price auction, this will have a close form, but for many others the DSP can have tables with approximate solutions that can be used instead of solving the problem in real time. Second,  many ad-exchanges use what are called hard reserve prices that consider a bid valid only if it is higher than the reserve price. This poses a problem for Theorem \ref{Thm:Applied} as the condition of $\rho(\cdot)$ being strictly non-decreasing would not be true. If the impression types had fixed hard reserve prices this is not a major issue as we can change the model to bid between the reserve price and $\bar{b}_i$ (if the reserve price were higher than $\bar{b}_i$ wouldn't bid for that impression type). In the case that hard reserve prices change dynamically, heuristics could be used, \textit{e.g.} considering bidding in real time only only for those campaigns with bid values higher than the reserve price, putting levels of reserve price as a field in the impression types, and others which we don't explore here. Third, it can be proven that Theorem \ref{Thm:Applied} guarantees that Algorithm \ref{Alg:TwoPhase} would converge to an optimal solution for \eqref{poi_deterministic} as we get better $\lambda$ solutions of the dual problem. For space reasons we don't extend on this topic here.

\section{Computational Experiments}\label{sec:expts}
Here we present computational results using the Ipinyou DSP data \cite{zhang2014real} to which we applied our two-phase solution procedure comparing its performance w.r.t. a greedy heuristic. The way we suggest and apply in our experiments an allocation and bidding variables $(\bx,\bb)$ in a practical RTB environment is shown in Policy \ref{Alg:OnlinePolicy} and the heuristic to which we compare our method is shown in Policy \ref{greedy-policy}.
\floatname{algorithm}{Policy}
\vspace{-2mm}
\begin{algorithm}
\caption{Online Policy Implied by $(\textbf{x,b})$
}\label{Alg:OnlinePolicy}
\begin{algorithmic}
\STATE {\bf Input:} Allocation and bidding variables $(\hat{\bx}, \hat{\bb})$ and a new impression arrival of type $i \in \calI$.
\STATE 1. Sample a campaign $\tilde{k} \in \calK_i$ according to the distribution implied by the values $x_{ik}$ for $k \in \calK_i$ or break with probability $1 - \sum_{k \in \calK_i} x_{ik}$ (\textit{i.e.} choosing no campaign). If some campaigns have depleted its budget adjust the probabilities to take this fact into account. 
\STATE 2. Enter bid price $\hat{b}_{i\tilde{k}}$. If the auction is won, then pay the ad exchange the amount is asking. If the auction is not won, then break.
\STATE 3. Show an ad for campaign $\tilde k$. If a click or action of interest happens, then deduct $q_{\tilde k}$ from the budget of campaign $\tilde k$ and earn revenue $q_{\tilde k}$.
\end{algorithmic}
\end{algorithm}
\vspace{-6mm}
\begin{algorithm}
\caption{Greedy Policy }\label{greedy-policy}
\begin{algorithmic}
\STATE {\bf Input:} New impression arrival $i \in \calI$.
\STATE 1. Bid for a campaign $k \in \calK_i$ with remaining budget bigger than $q_k$ with the highest $r_{ik}$ value.
\end{algorithmic}
\end{algorithm}
The Greedy Heuristic is optimal for the case of infinite budgets and second price auctions (is easy to extend it to arbitrary auction types). As an important remark our method should be used in a real DSP operation inside a Model Predictive Control Scheme, which calls Algorithm \ref{Alg:TwoPhase} as budgets get used and different model inputs gets updated as time progresses.

The public available Ipinyou data \cite{zhang2014real} contains information about real bidding made by the chinese DSP Ipinyou in 2013. It contains different features including the bidding prices of the impressions for which Ipinyou bid for, and the price paid by Ipinyou to the ad-exchange in case an impression was won and if a click or conversion occurred (we did not use conversion data). Ipinyou assumes that ad-exchanges use second price auctions. The data is already divided in train and test sets and it has been used to test bidding strategies for DSPs \cite{zhang2014optimal,ren2018bidding} but we haven't found a paper that use it for both bidding  and allocation strategies, reason why we compare to the Greedy Heuristic \ref{greedy-policy}. Ipinyou data is divided in three different time periods in 2013, of those we decided to use the third as in the first there is no information about the campaigns Ipinyou bid for, and in the second Ipinyou assumed that all impression types could serve all campaigns which make the impression-campaign graph non-interesting. The third season contains  3.15M of and 1.5M logs of impressions won by Ipinyou in the train and test set resp. in behlaf of four advertisers, which have 2716 and 1155 clicks associated to them. Here we use the different advertisers as our campaigns. 

To create ``impression types'' we divided the impressions by the visibility feature which has a strong correlation with CTR, and then by the regions, homepage url, and ``width x height'' of the ad to be shown (features that appear in all impression logs). The last three features have a high dimensionality, for example homepage url have 54,108 unique urls. For that reason we created mutually exclusive sets of the form all urls that were targeted only by advertiser 1, all that were targeted only by advertisers 1 and 3, etc. With this technique we partition all impressions for which Ipinyou bid for in 160 clusters of impressions which we used to create our final partition of 23 impression types. Of those 19 corresponds to the clusters with a minimum of 30,000 impressions won in the train set and the 4 left are the union of all clusters having different visibility attribute (we grouped together the second, third, fourth and fifth view as if they were the same visibility type). Our final graph is composed of 4 campaigns, 23 impression types, it has 43 edges, and the different CTR were taken as the empirical rates for each combination of (impression type, advertiser). Using only the impressions won in the train set for each impression type $i$ we fitted a beta distribution using the python Scipy package (imposing the location parameter to be equal to zero) to obtain parameters to estimate the bid landscape functions (the $\rho_i(\cdot)$ function is just a function call, but the $\beta_i(\cdot)$ function was estimated using Monte-Carlo). Finally, we count the times each impression type appears in the test set to create the $s_i$ values, and the budgets correspond to the total amount of money that Ipinyou paid for the impressions assigned to each advertiser in the test set. To simulate a real time environment we used the empirical train CTR to train our models and the greedy heuristic (we took the average CTR per impression for the pairs of (impression type,advertiser) that appear less than 5000). To test our model we use the impressions won by Ipinyou in the actual order saving their impression type. Then, one by one we read each impression log and we assume that the impression was won if the proposed bid for it is higher than the amount Ipinyou paid for it. A click for the proposed advertiser occurs with probability equal to the empirical CTR from the test data for the pair (impression type, advertiser). We tried the utility functions 1. and 2. from Section \ref{sec:optalg}, using $\tau_k=1/m_k$ for all $k \in \calK$ for second one.
\vspace{-2mm}
\paragraph{Results} Our results are shown in Figure \ref{thefig}. Here we define budget utilization (b.u.) as the percentage of the total budget that was used at the end of one simulation, and u.f. stands for utility function. We performed two experiments. The first was to see the sensitive of our model w.r.t. to the budget. We tried $1/32$, $1/8$, $1/4$, $1/2$ and $1/0$ of the budgets Ipinyou used for each advertiser in the test data and we run 100 simulations for each setting obtaining average profit and b.u. Our results are all relative to the average profit and b.u. obtained by using the greedy heuristic. What makes one simulation different from the other is that the CTR is a random variable. In the second experiment we multiply the penalization parameters $tau_k=1/m_k$ that appear in the u.f. 2  by 0.1, then by 0.3, and so on until 2.1 running 100 simulations in each case obtaining the average profit and b.u. All our results are relative to the average profit and b.u. of the u.f. 1. and we also include the average results from the greedy heuristic for comparison.  Our results show that our methodology works very well for cases in which the budget is tight, but when is not the case the greedy heuristic is a good alternative. From our second experiment we can see that a better b.u. utilization can be obtained at the cost of having a worst profit (it can even be negative).
\begin{figure}[h!]
\vspace{-6mm}
\begin{multicols}{2}
\scalebox{0.66}{
\begin{tabular}{c}
\begin{tabular}{cccccc}
                    \multicolumn{6}{c}{{\bf Relative Profit}}  \\ \hline
{\bf (Perc./Met.)}	   & {\bf 1/32} & {\bf 1/8} & {\bf 1/4} & {\bf 1/2} & {\bf 1.0} \\ \hline
{\bf U.F. 1}  & {\Large 2.10}     & {\Large 1.38}  & {\Large 1.08}     & {\Large 0.92}   & {\Large 1.21}    \\
{\bf U.F. 2}    & {\Large 1.83}     & {\Large 1.34}  & {\Large 0.94}     & {\Large 0.76}   & {\Large 0.8}     \\\hline
\end{tabular}
\\ \\ 
\begin{tabular}{cccccc}
                    \multicolumn{6}{c}{{\bf Relative Budget Utilization}}  \\ \hline
{\bf (Perc./Met.)}	   & {\bf 1/32} & {\bf 1/8} & {\bf 1/4} & {\bf 1/2} & {\bf 1.0} \\ \hline
{\bf U.F. 1}  & {\Large 0.95}     & {\Large 0.88}  & {\Large 0.8}     & {\Large 0.95}   & {\Large 1.07}    \\
{\bf U.F. 2}    & {\Large 0.84}     & {\Large 0.94}  & {\Large 1.07}     & {\Large 1.10}   & {\Large 1.38}     \\\hline
\end{tabular}
\end{tabular}
}
\vfill\null
\columnbreak
\includegraphics[scale = 0.24]{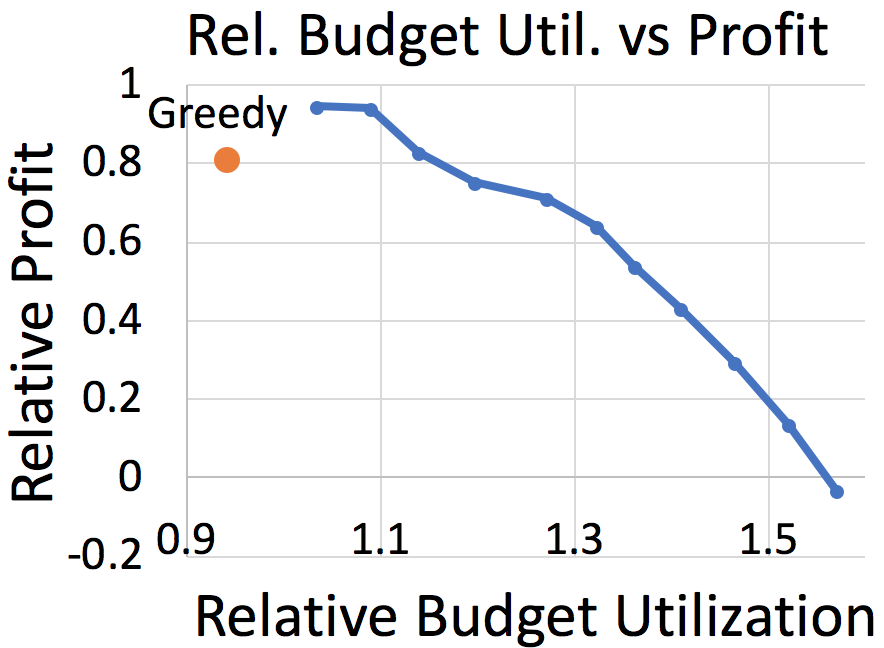}
\end{multicols}
\vspace{-8mm}
\caption{Two-phase policy vs. greedy policy}
\label{thefig}
\vspace{-2mm}
\end{figure}

Let us conclude this section by mentioning a few directions for future research. It would be very valuable to perform experiments in which impressions are auctioned in both first and second price auctions and in which the cardinality of the impression types and campaigns are higher. Also, several of the quantities we assume as known in this work are hard to estimate in practice. We will study robust approaches to our model.

\section{Proofs of the Propositions and Theorem 5.1}\label{Sec:Proofs}

In this appendix we will make use of the following definitions which for space we skip in the main text (we use them in the proofs of Proposition \ref{Prop:SubgQ} and Theorem \ref{Thm:Applied}):
\begin{definition}\label{Def:Appendix}
\begin{itemize}
\item The set of feasible bidding and allocation variables for a given impression type is: 
$$S_i := \left\lbrace (x_i,b_i) \in \mathbb{R}^{|\calK_i|} \times \mathbb{R}^{|\calK_i|}: \ \ b_i \in [0,\bar{b}_i]^{|\calK_i|}, \ \ \sum_{k \in \calK_i} x_{ik} \le 1, x_{i} \ge 0  \right\rbrace$$
\item The expected amount campaigns spend in impressions of type $i$ is $v_i(\bx,\bb) \in \mathbb{R}^{|\calK|}$ which is a vector whose $k \in \calK$ coordinate takes the value of $r_{ik}s_ix_{ik} \rho_i(b_{ik})$ if $(i,k) \in \calK$ and $0$ o.w. Then, $\sum_{i \in \calI} v_i(\bx,\bb)= v(\bx,\bb)$.
\item The contribution of the different impression types in the dual function is separable, and therefore given a fixed dual variable $\lambda \in \mathbb{R}^{|\calK|}$ we define for each $i \in \calI$:
$$\psi_i(\lambda) = \underset{(x_i,b_i) \in S_i}{\max} \ \ \sum_{k \in \calK_i} h_i(r_{ik}(1-\lambda_k),b_{ik})s_i x_{ik} $$
\item Let's define the profit from impressions of type $i$ in the objective from impression of type $i$ in \eqref{old-poi} as $\pi_i(\bx,\bb) = \sum_{k \in \calK_i} h_i(r_{ik},b_{ik}) s_i x_{ik}$, and therefore $\sum_{i \in \calI} \pi_i(\bx,\bb)=\pi(\bx,\bb)$.
\end{itemize}
\end{definition}

Notice that for any $i \in \calI$ both $v_i(\bx,\bb)$ and $\pi_i(\bx,\bb)$ depend only in $(x_i,b_i) \in S_i$, but we decide to use this notation to not carry the ``$\sum_{i \in \calI}$'' everywhere. 

\subsection{Proof of Proposition \ref{Prop:MaxQ}}
\begin{proof}
For fixed $\lambda$ we have that $p(\lambda)$ is a constant, then we need to focus only on the maximization part of $Q(\lambda)$. For any fixed $\bf{x}' \ge 0$ we have:
\begin{align*}
& \underset{(\textbf{x}',\textbf{b}) \in S}{\max} \ \ \sum_{k \in \calK} \sum_{i \in \calI} h_i(r_{ik}(1-\lambda_k),b_{ik}) s_ix_{ik}' = \\
& \sum_{k \in \calK} \sum_{i \in \calI} \underset{b_{ik} \in [0,\bar{b}_i]}{\max} \ \   h_i(r_{ik}(1-\lambda_k),b_{ik}) s_ix_{ik}'
\end{align*}
Take $(i,k) \in \calE$ arbitrary and assume $x_{ik}' \ge 0$, then  any $b(\lambda)_{ik} \in \arg\max_{b \in [0,\bar{b}_i]}  h_i(r_{ik}(1-\lambda_k),b_{ik}) s_i$ is an optimal  bidding price for the pair $(i,k)$ independent of the value of $x_{ik}'$ which proves STEP 1. of Algorithm \ref{Alg:MaxQ}. 
Fixing $\bb = \bb(\lambda)$, we can use that for each $i \in \calI$ we have $\sum_{k \in \calI_k} x_{ik} \le 1$ and $x_{ik}\ge0$ for all $(i,k) \in \calE$ are the only constraints for $\bx$ (w.r.t. to impression type $i$), therefore is optimal for each impression type $i \in \calI$ to bid only for a campaign that maximizes the profit we can get from it. If for an impression type $i$ there is no campaign $k \in \calK_i$ which gives us a positive profit, then is optimal to not bid for that impression type. That's exactly what STEP 2. of Algorithm \ref{Alg:MaxQ} concluding the proof.
\end{proof}

\subsection{Proof of Proposition \ref{Prop:SubgQ}}
\begin{proof}
Here we assume $\lambda \in \mathbb{R}^{|\calK|}$ fix. Notice that using the definitions from Definition \ref{Def:Appendix} we have  $Q(\lambda) = \sum_{i \in \calI}  \psi_i(\lambda) + p(\lambda)$. We are going that to show that for any $(\bx(\lambda),\bb(\lambda)) \in \arg \max Q(\lambda)$ and any $\lambda' \in \mathbb{R}^{|\calK|}$ we have $\psi_i(\lambda') \ge \psi_i(\lambda) - v_i(x_i(\lambda),b_i(\lambda))^T(\lambda' -\lambda)$. The previous is enough to finish this proof as $\partial Q(\lambda) =  \partial \sum_{i \in \calI} \psi_i(\lambda) +\partial p(\lambda) = \sum_{i \in \calI} \partial  \psi_i(\lambda) +\partial p(\lambda)$ and $\sum_{i \in \calI} v_i(x_i(\lambda),b_i(\lambda)) = v(\bx(\lambda),\bb(\lambda))$. Let $\lambda' \in \mathbb{R}^{|\calK|}$ be any dual variable different from $\lambda$ (if not the following set of equations are trivial), and let $i \in \calI$ arbitrary, then:
\begin{align*}
\psi_i(\lambda') = & \underset{(x_i,b_i) \in S_i}{\max} \sum_{k \in \calK_i}  \left( r_{ik}(1-\lambda_k')  -\beta_i(b_{ik}) \right)\rho_i(b_{ik}) s_i x_{ik} \\
= & \underset{(x_i,b_i) \in S_i}{\max} \sum_{k \in \calK_i} \left( r_{ik}(1-\lambda_k) -r_{ik}(\lambda_k' -\lambda_k )  -\beta_i(b_{ik}) \right)\rho_i(b_{ik}) s_i x_{ik} \\
\ge & \sum_{k \in \calK_i} h_i(r_{ik}(1-\lambda_k),b_{ik}(\lambda) )  s_i x_{ik}(\lambda) \\
& - \sum_{k \in \calK_i} (\lambda_k' -\lambda_k )r_{ik}\rho_i(b_{ik}(\lambda) )s_i x_{ik}(\lambda)  \\
\ge &  \psi_i(\lambda) - v_i(x_i(\lambda),b_i(\lambda))^T(\lambda' -\lambda)
\end{align*}
\end{proof}

\subsection{Proof of Proposition \ref{Prop:FbConvex}}
\begin{proof}
Assume $\textbf{b}$ fixed. The domain of $F(\textbf{b})$ is a set of linear constraints, then we only need to proof that the objective function is concave. Clearly the summation part of the objective function is linear and therefore concave, then is only left to show that the utility function part is concave. The latter is true as  $v(\cdot,\bb)$ is an affine function in term of the allocation probabilities, therefore $u(v(\cdot,\bb))$ is the composition of a concave and affine function and therefore concave.
\end{proof}

\subsection{Proof of Theorem \ref{Thm:Applied}}

This proof uses the terminology defined in Definition \ref{Def:Appendix} and we can assume $\rho_i(\cdot)$ and $\beta_i(\cdot)$ to be continuous functions for all parts in this proof as it is a weaker condition than being differentiable which is part of the hypothesis of Theorem \ref{Thm:Applied}.  We will start by showing that if we assume $\rho_i(\cdot)$ and $\beta_i(\cdot)$ to be continuous for all $i \in \calI$, then we can obtain the explicit form of the sub-differential $\partial Q(\lambda)$ for any fixed $\lambda \in \mathbb{R}^{|\calK|}$ (Lemma \ref{Lem:Danskins}). Result that we combine with Lemma \ref{Lem:FirstND}  to prove Lemma \ref{Cor:SufSD}. The later results will help us prove that if there exists an optimal dual variable $\lambda^*$ that satisfies a condition we call ``Unique Solution'' (Definition \ref{Def:US}), then there is no duality gap (Lemma \ref{Thm:UniqSoltn}). Finally, we will prove Theorem \ref{Thm:Applied} by showing that when the hypothesis of the theorem holds, then the ``Unique Solution'' condition hold for any feasible lambda $\lambda$ (\textit{i.e.}, $p(\lambda)< \infty$), and therefore for any optimal dual variable.

Let's first define $S_i^*$ as the set of optimal solutions for $\psi_i(\cdot)$ given $\lambda \in \mathbb{R}^\calK$ for some $i \in \calI$, \textit{i.e.}
$S_i^*(\lambda) : = \arg \underset{(x_i,b_i) \in S_i}{\max} \psi_i(\lambda) $, and $S^*= S_1^*(\lambda) \times \dots S_{|\calI|}^*(\lambda)$.

\begin{lemma} \label{Lem:Danskins} \textbf{(Equivalent to Lemma 3.3 in \cite{wang2017vanishing})} 
The sub-differential of the dual function is 
\begin{align}
\partial Q(\lambda) = \textbf{conv}\left \lbrace - \sum_{i \in \calI} v_i(\bx,\bb) +y  \Big| \text{ } (x_i,b_i) \in S_i^*(\lambda), y \in \partial p(\lambda)    \right \rbrace \nonumber
\end{align}
\end{lemma}
\begin{proof}
Let's define $\phi_i(\lambda, (x_i,b_i)): = $ $\pi_i(x_i,b_i)-\lambda^T v_i(\bx,\bb)$ for each $i \in \calI$. Then, $\psi_i(\lambda) = \underset{(x_i,b_i) \in S_i}{\max}$ $\phi_i(\lambda,(x_i,b_i))$, and $\phi(\cdot,\cdot)$ is differentiable w.r.t. to its first argument,  $\phi(\cdot,\cdot)$ is a continuous function w.r.t. both of its arguments,  and $\partial \phi / \partial \lambda$ is a continuous function w.r.t. to its second argument  for all $\lambda$. Then, Danskin's Theorem says $\partial \psi_i(\lambda) := \textbf{conv}\{ -v_i(\bx,\bb) |\text{ } (x_i,b_i) \in S_i^*(\lambda)  \}$. Finally, using that $Q(\lambda) =  \sum_{i \in \calI} \psi_i(\lambda) + p(\lambda)$ we obtain the desired conclusion.
\end{proof}

\begin{lemma} \label{Lem:FirstND}
If there exists $\lambda^*$ optimal dual variable, such that $v(\bx^*,\bb^*) \in \partial p(\lambda^*)$, for some $(\bx^*,\bb^*) \in S^*(\lambda^*)$, then $Q(\lambda^*) = F(\textbf{x}^*,\textbf{b}^*)$. 
\end{lemma}
\begin{proof} 
Using that $Q(\lambda^*)=Q^*$, the optimality of $(\bx^*,\bb^*)$ and the definition of $F(\bx,\bb)$ we have:
	\begin{align}
	F(\textbf{x}^*,\textbf{b}^*) = \ \  & \pi(\bx^*,\bb^*) +\inf_{\lambda \in \calR^\calK}  \left\lbrace p(\lambda) -\lambda^T v(\bx^*,\bb^*) \right\rbrace \nonumber\\
	 Q(\lambda^*) = \ \  & \pi(\bx^*,\bb^*) + p(\lambda^*) - v(\bx^*,\bb^*)^T  \lambda^* \nonumber
	\end{align}
	We will show that $F(\textbf{x}^*,\textbf{b}^*) = Q(\lambda^*)$ by proving 
    $$\underset{\lambda \in \calR^\calK}{\inf} \left\lbrace p(\lambda) -\lambda^T v(\bx^*,\bb^*) \right\rbrace= p(\lambda^*) - v(\bx^*,\bb^*)^T \lambda^* $$ 
    Let's define the convex function $\zeta(\lambda) : = p(\lambda) - v(\bx^*,\bb^*)^T \lambda$, then  $\left( g- v(\bx^*,\bb^*) \right) \in \partial \zeta(\lambda^*)$ for any $g \in \partial p(\lambda^*)$. By hypothesis it exists $g' \in \partial p(\lambda^*)$, such that $g' - \sum_{i \in \calI} v_i(x_i^*,b_i^*) = 0$ then,  using the subgradient inequality:
    $$\zeta(\lambda) \ge \zeta(\lambda^*) + \left(g'- v(\bx^*,\bb^*) \right)^T (\lambda -\lambda^*) = \zeta(\lambda^*)$$
    Which shows that $\lambda^*$ is a minimizer of $\zeta(\cdot)$ concluding the proof.
\end{proof}

\begin{lemma}\label{Cor:SufSD}
If there exists $\lambda^*$ optimal dual solution,  such that the sets $\mathcal{V}_i := \{ v_i(\bx^*,\bb^*) | (x_i^*,b_i^*) \in S_i^*(\lambda^*)  \}$ are convex for all $i \in \calI$, then it exists $(\textbf{x}^*,\textbf{b}^*) \in \arg \max Q(\lambda^*)$, such that $F(\textbf{x}^*,\textbf{b}^*)= Q(\lambda^*)$.
\end{lemma}
\begin{proof}
If $\mathcal{V}_i$ is convex by definition we have $\mathcal{V}_i = conv(\mathcal{V}_i)$ for all $i \in \calI$ which would imply $\partial \sum_{i \in \calI} \psi_i(\lambda)= \sum_{i \in \calI} \mathcal{V}_i$. 
Lemma \ref{Lem:Danskins} tells us $-\partial \sum_{i \in \calI} \psi_i(\lambda) +\partial p(\lambda^*) = \partial Q(\lambda^*)$, and by the optimality of $\lambda^*$ we have that $0 \in \partial Q(\lambda^*)$. Using the convexity of the $\mathcal{V}_i$, there exists $(x_i^*,b_i^*) \in S_i^*(\lambda^*)$ for all $i \in \calI$ and $g \in \partial p(\lambda^* )$, such that if we call $(\bx^*,\bb^*)=((x_1^*,b_1^*),\dots,(x_{|\calI|}^*,b_{|\calI|}^*))$ we have $v(\bx^*,\bb^*) =g$ and then Lemma \ref{Lem:FirstND} holds.
\end{proof}

\begin{definition}[Condition: Unique Solution (US)]\label{Def:US}
We say that the dual variable $\lambda$ satisfy the unique solution condition if \\ $ \arg\max_{b \in [0,\bar{b}_i]}  h_i(r_{ik}(1-\lambda_k),b)$ is a singleton  whenever $\lambda_k<1$ for all $(i,k) \in \calE$.
\end{definition}

\begin{lemma}[Unique Solution] \label{Thm:UniqSoltn} Assume $\beta_i(b) \rho_i(b) >0$ for all $b \in (0,\bar{b}_i]$, then if $\lambda^*$ is an optimal dual solution that satisfies the Unique Solution condition, then there exists $(\textbf{x}^*,\textbf{b}^*) \in \arg \max Q(\lambda^*)$ such that $Q(\lambda^*) = F(\textbf{x}^*,\textbf{b}^*)$. 
\end{lemma}
\begin{proof}[Proof of Theorem \ref{Thm:UniqSoltn}]\label{Proof:UniqSoltn}
Let $\lambda^*$ be an optimal dual solution that satisfies the Unique Solution condition. Our goal here will be to show that sets $\mathcal{V}_i$ from Corollary \ref{Cor:SufSD} w.r.t. to $\lambda^*$ are convex for all $i \in \calI$ which would conclude this proof. Let $\calI= \calI' \cap \calI''$ with $\calI'$ the set of impression types such that $\psi_i(\lambda^*) >0$ for all $i \in \calI'$, and $i \in \calI''$ if $\psi_i(\lambda^*) =0$. Notice that we can assume $r_{ik}>0$ for all $(i,k) \in \calE$ as if it weren't case for some $(i,k) \in \calE$ it would be optimal to make $x_{ik}=0$ in our primal problem \eqref{old-poi} and the edge could have been removed from the problem definition. Then, using the latter, the continuity of functions $\rho_i(\cdot)$ and $\beta_i(\cdot)$ and that $\rho_i(b)\beta_i(b)>0$ for all $b>0$ and $i \in \calI$ an impression type $i$ is in $\calI''$ if and only if $\lambda_k \ge 1$ for all $k \in \calK_i$, and $b_{ik}(\lambda_k^*)=0$ is the unique maximizer of $\underset{b \in [0,\bar{b}_i]}{\max}$ $h_i(r_{ik}(1-\lambda),b)$ for all $k \in \calK_i$. For this case we have:
$$S_i(\lambda^*) = \left\lbrace (x_i,b_i) \in S_i \Big| \ \  \begin{aligned}
 x_{ik} b_{ik}=0     && \text{ if } \lambda_k^*=1 \\
  x_{ik} \rho_i(b_{ik})=0 && \text{ if } \lambda_k^*>1
\end{aligned}   \right\rbrace$$
Then, the set $\mathcal{V}_i$ from Lemma \ref{Cor:SufSD} is a convex set for any $i \in \calI''$ as it would be equal to the convex hull over a finite amount of points. This is easy to see as the $k^{th}$ coordinate of any of its vectors is equal to 0 if $\lambda_k^*>1$ or if $k \notin \calK_i$, and if $k \in \calK_i^*$ and $\lambda_k^*=1$ the general form of the $k^{th}$ coordinate of its vectors is $r_{ik}\rho_i(0)s_i x_{ik}$ and $\sum_{i \in \calK_i} x_{ik} \le 1$ with $x_{ik} \ge 0$ for all $k \in \calK_i$. For $i \in \calI'$ there is at least one $k \in \calK_i$ with $\lambda_k<1$ which implies $h_i(r_{ik}(1-\lambda_k^*),b_{ik}(\lambda^*))>0$ (the implication follows from the continuity of the bid landscape and $\beta_i(b)\le b$ for all $b>0$ and $i \in \calI$). Let's define $\calK_i^* \subseteq \calK_i$ as $ \calK_i^*: =\arg\underset{k \in \calK_i}{\max}\left\lbrace  h_i(r_{ik}(1-\lambda_k^*),b_{ik}(\lambda_k^*))  s_i   \right\rbrace$ for all $i \in \calI' $, then for any $i \in \calI'$:
\begin{align*}
S_i(\lambda^*) = \Big\lbrace (x_i,b_i) \in S_i \Big| \text{ } & \sum_{k\in \calK_i^*} x_{ik} =1, \text{ } x_{ik}=0 \text{ for all } k \notin \calK_i^*, \\
&  b_{ik} = b_{ik}(\lambda_k^*)  \text{ for all } k \in \calK_i^*  \Big\rbrace
\end{align*}
Making $\mathcal{V}_i$  a convex set for all $i \in \calI'$.
\end{proof}

Now we will use Lemma \ref{Thm:UniqSoltn} to prove Theorem \ref{Thm:Applied} concluding this section.

\begin{proof}[Proof Theorem \ref{Thm:Applied}] \label{ProofThmApplied}

First notice that the second condition of Theorem \ref{Thm:Applied} implies that $\beta_i(b)\rho_i(b)>0$ for all $i \in \calI$. Let $\rho(\cdot)$ and $\beta(\cdot)$ denote arbitrary bid landscape functions and $[0,\bar{b}]$ the bidding range for an arbitrary impression type. Let $r>0$ arbitrary represent and expected revenue, then for any $b \in (0,\bar{b}]$ we have:
\[   
\frac{\partial h(r,b)}{\partial b}:=h'(r,b)=\rho'(b)(r-g(b)) = 
     \begin{cases}
       h'(r,b)>0 & \text{ if } g(b)<r  \\
       h'(r,b)<0 & \text{ if } g(b)>r  \\
        h'(r,b)=0 & \text{ if } g(b)=r 
     \end{cases}
\] 
Notice that $\rho'(b)>0$ by hypothesis so $h'(r,b)$ sign is directly related to the term $(r-g(b))$. Let's use that $g(\cdot)$ is strictly increasing in $(0,\bar{b}]$ and define $\hat{b}(r) := \liminf \{b \in (0,\bar{b}]:\text{ }  g(r) > r\}$. Then, $b^*(r)= \arg\underset{b \in [0,\bar{b}]}{\max} $ $h(r,b)$ is equal to: 
\[   
b^*(r) = 
\begin{cases}
\bar{b} & \text{ if } g(\bar{b}) \le r  \\
\hat{b}(r) & \text{o.w.} 
\end{cases}
\] 
The previous can be understood in very simple words, given  that $g(\cdot)$ is strictly increasing in $(0,\bar{b}]$ if $g(\bar{b}) \le r$, then is optimal to bid the maximum possible which is $\bar{b}$ as $h'(r,b)$ is always strictly positive in $(0,\bar{b}]$. If $g(b) \ge r$ for all $b >0$ then $h'(r,b)<0$ for all $(0,\bar{b}]$ and therefore is optimal to bid 0 and by definition we would have $\hat{b}(r)=0$. Finally, if neither of the previous cases occurs it means that there was a change of sign of $h'(r,b)$ in $(0,\bar{b}]$ making it optimal to bid the value in which the change of sign occurs which is $\hat{b}(r)$.  Here we need to use $\lim\inf$ in the definition of $\hat{b}(r)$ as $g(\cdot)$ is not necessarily continuous. Notice that $\hat{b}(r)$ always exists and is unique by definition. Notice that we have proven that the Unique Solution would hold for any dual variables, and therefore for any optimal dual variable $\lambda^*$. Then, by Lemma \ref{Thm:UniqSoltn} holds we conclude this proof.
\end{proof}

\bibliographystyle{ACM-Reference-Format}
\bibliography{reference} 

\end{document}